\documentclass[reqno]{amsart}

\usepackage{amssymb, amsmath, amsthm}

 \newtheorem{thm}{Theorem}[section]
  \newtheorem{cor}[thm]{Corollary}
 \newtheorem{lem}[thm]{Lemma}
 \newtheorem{prop}[thm]{Proposition}

    \theoremstyle{definition}

    \newtheorem{notn}[thm]{Remark}
  \newtheorem{example}[thm]{Example}
  
 \theoremstyle{remark}
 
\numberwithin{equation}{section}

\DeclareMathOperator{\ham}{ham}

\begin{document}

\title[Enveloping algebras of double Poisson-Ore extensions]{Enveloping algebras of double Poisson-Ore extensions}

\author{Jiafeng L\"u}
\address{Department of Mathematics, Zhejiang Normal University, Jinhua, Zhejiang 321004, China} \email{jiafenglv@zjnu.edu.cn}
\author{Sei-Qwon Oh}
\address{Department of Mathematics, Chungnam National  University, 99 Daehak-ro,   Yuseong-gu, Daejeon 34134, Korea} \email{sqoh@cnu.ac.kr}
\author{Xingting Wang}
\address{Department of Mathematics, Temple University, Philadelphia, 19122, USA} \email{xingting@temple.edu}
\author{Xiaolan Yu}
\address{Department of Mathematics, Hangzhou Normal University, Hangzhou, Zhejiang 310036, China}\email{xlyu@hznu.edu.cn}


\subjclass[2010]{17B63, 16S10}

\keywords{Double Ore extension, double Poisson-Ore extension, Poisson enveloping algebra}



\begin{abstract}
It is proved that the Poisson enveloping algebra of a double Poisson-Ore extension is an iterated double Ore extension.
As an application, properties that are preserved under iterated double Ore extensions are invariants of the Poisson enveloping algebra of a double Poisson-Ore extension.
\end{abstract}

\maketitle


\section*{Introduction}

Let $R$ be a Poisson algebra. In  \cite{Oh5}, the second author constructed an associative algebra $R^e$, called the Poisson enveloping algebra of $R$,  in order that
the category of Poisson modules over $R$ is equivalent to that of modules over $R^e$. Since then the subject has been developed in \cite{Umi} and \cite{LuWaZh}.
In particular, the first, the third authors and Zhuang  studied the Poisson enveloping algebra of a Poisson-Ore extension of $R$ in \cite{LuWaZh2} and showed that it is an iterated Ore extension of $R^e$ and inherits algebraic properties of $R^e$ including noetherianess, Artin-Schelter regularity and etc. On the other hand,  in \cite{LoOhWa},  Lou, Wang and the second author gave a notion of double Poisson-Ore extension arising from the semi-classical limit of certain double Ore extension, which can be thought as a generalized Poisson-Ore extension with two variables. This motivates us to show that the Poisson enveloping algebras of  double Poisson-Ore extensions have algebraic properties similar to those obtained in \cite{LuWaZh2}.

 In the section 1, we modify the construction of Poisson enveloping algebra given in \cite[\S5]{LuWaZh} to be understood easily. Namely, let $R$ be any Poisson algebra over a basis field ${\bf k}$ and let $\Omega_{R/{\bf k}}$ be the K$\ddot{a}$hler differential
of $R$. Then it is observed that $\Omega_{R/{\bf k}}$ is a Lie algebra with Lie bracket induced by the Poisson bracket. Hence there exists a semi-crossed product $R\rtimes U(\Omega_{R/{\bf k}})$ of $R$, where  $U(\Omega_{R/{\bf k}})$ is the universal enveloping algebra of the Lie algebra $\Omega_{R/{\bf k}}$.
We see in Proposition~\ref{PENV} that the Poisson enveloping algebra $R^e$ is a quotient algebra of $R\rtimes U(\Omega_{R/{\bf k}})$ by certain ideal.
 Let $A=R[x_1,x_2]_p$ be a
double Poisson-Ore extension of $R$. In the section 2, we obtain a Poisson enveloping algebra $A^e$ by using the result
of the section 1 and find a valuable filtration $\mathcal{F}$ by giving suitable degrees
on each canonical generators of $A^e$.  Finally we
prove in Theorem~\ref{ENVDPE} that $A^e$ is an iterated double Ore extension by using algebraic properties of the  graded algebra associated to $\mathcal{F}$. The method of using the filtration $\mathcal{F}$ makes us avoid tedious computations in \cite[\S2]{LuWaZh2} as observed in Remark~\ref{RemarkGR}.
As an application of the fact that $A^e$ is an iterated double Ore extension of $R^e$, we induce, under certain conditions, invariants of algebraic properties in Corollary~\ref{IteratedORE}, Corollary~\ref{ENVDPE2} and Corollary~\ref{GRENVDPE} including maximal order, Artin-Schelter regular algebra, Calabi-Yau algebra, Koszul algebra and Auslander-Gorenstein algebra, which are true in most of the examples we are interested.

\medskip
Assume throughout the paper that ${\bf k}$ denotes a base field of characteristic zero, that all vector spaces are over ${\bf k}$ and that all algebras have unity.
A {\it Poisson algebra} is a commutative ${\bf k}$-algebra $R$ with a Poisson bracket, that is a bilinear map
$\{-,-\}:R\times R\rightarrow R$ such that $R$ is a Lie algebra
under $\{-,-\}$ and, for all $a\in R$, the hamiltonian map $\ham(a):=\{a,-\}$ is a derivation of $R$, which is called {\it Leibniz rule}.

For an algebra $A$, we denote by $A_L$  the Lie algebra $A$ with Lie bracket $[a,b]:=ab-ba$ for $a,b\in A$.

\section{Poisson enveloping algebra}

For the clearance of the structure of a Poisson enveloping algebra, we will modify the construction of Poisson enveloping algebra given in \cite[\S5]{LuWaZh}, which will be used in the next section.

Let $R$ be a Poisson algebra and $U$ be an algebra. For an algebra homomorphism $\alpha:R\longrightarrow U$ and a Lie algebra homomorphism $\beta:(R,\{-,-\})\longrightarrow U_L$, the pair $(\alpha,\beta)$ is said to {\it satisfy the  property ${\bf P}$ from $R$ into $U$} if $\alpha$ and $\beta$ satisfy the following properties: for all $a,b\in R$,
$$\begin{aligned}
\alpha(\{a,b\})&=[\beta(a),\alpha(b)], &
\beta(ab)&=\alpha(a)\beta(b)+\alpha(b)\beta(a).
\end{aligned}$$

Recall the definition of Poisson enveloping algebra in \cite[Definition 3]{Oh5}.
  A triple $(U, \alpha,\beta)$,
where $U$ is an algebra and the pair $(\alpha,\beta)$ satisfies the property {\bf P} from  a Poisson algebra $R$ into $U$,
 is called the {\it Poisson enveloping algebra} of $R$
if the following universal property holds:   For any triple $(A,\gamma, \delta)$ such that $A$ is an algebra and the pair $(\gamma,\delta)$ satisfies the property {\bf P} from $R$ into $A$, there exists a unique algebra homomorphism $h$ from $U$ into $A$
such that $h\alpha=\gamma$ and $h\beta=\delta$. The algebra homomorphism $\alpha$ is a monomorphism by \cite[Proposition 2.2]{OhPaSh2} and the Poisson enveloping algebra of any Poisson algebra exists uniquely up to isomorphism by \cite[Theorem 5]{Oh5}. We will denote by $R^e$ the Poisson enveloping algebra of $R$.

Given a Poisson algebra $R$, let  $F$ be a free left $R$-module with basis $\{\text{d}r|r\in R\}$ and $M$ be a submodule of $F$ generated by the elements
\begin{equation}\label{Kahl}
\begin{aligned}
&\text{d}(r_1+r_2)-\text{d}r_1-\text{d}dr_2,\\
&\text{d}(r_1r_2)-r_1\text{d}r_2-r_2\text{d}r_1,\\
&\text{d} a
\end{aligned}
\end{equation}
for all $a\in{\bf k}$ and $r_1,r_2\in R$. Then the K$\ddot{a}$hler differential module of $R$ is
$$\Omega_{R/{\bf k}}:=F/M.$$
The induced map
$$\text{d}:R\longrightarrow \Omega_{R/{\bf k}},\ \ a\mapsto \text{d}a$$
is a derivation by (\ref{Kahl}).

Let $H=(H,m,u,\Delta,\epsilon,S)$ be a Hopf algebra. An algebra $A$ is said to be a {\it left $H$-module algebra} if $A$ is a left $H$-module satisfying
$$h\cdot(ab)=\sum_{(h)}(h_1\cdot a)(h_2\cdot b),\ \ h\cdot 1=\epsilon(h)1,\ \ \ \ h\in H,\ \  a,b\in A,$$
where $1$ is the unity of $A$ and $\Delta(h)=\sum_{(h)} h_1\otimes h_2$.  If $A$ is a left $H$-module algebra then there exists an algebra $A\otimes_{\bf k} H$ with multiplication
$$(a\otimes h)(b\otimes g)=\sum_{(h)}a(h_1\cdot b)\otimes h_2g,\ \  \ \ a,b\in A,\ \ h,g\in H$$
by \cite[Proposition 1.6.6]{Maj}. Such an algebra is called a {\it semi-crossed product} of $A$ and $H$ and denoted by $A\rtimes H$.

By \cite[Example 5.4]{LuWaZh}, $\Omega_{R/{\bf k}}$ is a Lie algebra over ${\bf k}$ with Lie bracket
\begin{equation}\label{Kahl2}
[a\text{d}r,b\text{d}s]=ab\text{d}\{r,s\}+a\{r,b\}\text{d}s-b\{s,a\}\text{d}r
\end{equation}
for all $a, b,r,s\in R$.
Let $U(\Omega_{R/{\bf k}})$ be the corresponding universal enveloping algebra.
Note that $U(\Omega_{R/{\bf k}})$ is a Hopf algebra with Hopf structure
$$\Delta(a\text{d}r)=a\text{d}r\otimes 1+1\otimes a\text{d}r,\ \ \epsilon(a\text{d}r)=0,\ \ S(a\text{d}r)=-a\text{d}r$$
for all $a,r\in R$.
Let us show that $R$ is a left $U(\Omega_{R/{\bf k}})$-module algebra.
For $a\text{d}r\in \Omega_{R/{\bf k}}$ and $b\in R$, define
\begin{equation}\label{ACT}
a\text{d}r\cdot b=a\{r,b\},
\end{equation}
which  is well-defined with respect to the relations (\ref{Kahl}).
The action (\ref{ACT}) makes $R$  a left $\Omega_{R/{\bf k}}$-module and thus $R$ is a left $U(\Omega_{R/{\bf k}})$-module.
Since every element of $\Omega_{R/{\bf k}}$ acts as a derivation on $R$, $R$ is a left $U(\Omega_{R/{\bf k}})$-module algebra.
It follows that there exists the semi-crossed product $R\rtimes U(\Omega_{R/{\bf k}})$, as observed in the above paragraph. That is, $R\rtimes U(\Omega_{R/{\bf k}})$ is the algebra $R\otimes_{\bf k} U(\Omega_{R/{\bf k}})$ with  multiplication
\begin{equation}\label{Kahl3}
(a\otimes f)(b\otimes g)=\sum_{(f)}a(f_1\cdot b)\otimes f_2g.
\end{equation}
for $a,b\in R$ and $f,g\in U(\Omega_{R/{\bf k}})$.

Note that $\Omega_{R/{\bf k}}$ is a Lie algebra over ${\bf k}$ as well as a left $R$-module and that $\Omega_{R/{\bf k}}\subset U(\Omega_{R/{\bf k}})$. Hence
$$U(\Omega_{R/{\bf k}})={\bf k}+\Omega_{R/{\bf k}}+(\Omega_{R/{\bf k}})^2+(\Omega_{R/{\bf k}})^3+\ldots$$
as ${\bf k}$-vector spaces and the subspace $U^{\geq1}:=\sum_{i\geq1}(\Omega_{R/{\bf k}})^i$ of $U(\Omega_{R/{\bf k}})$ is a left $R$-submodule.
\begin{lem}
Let $a,b,r,s\in R$. In $U(\Omega_{R/{\bf k}})$,
\begin{equation}\label{LEM1}
(a\text{d}r)(b\text{d}s)=ab\text{d}r\text{d}s+a\{r,b\}\text{d}s.
\end{equation}
\end{lem}

\begin{proof}
We have that
$$\begin{aligned}(a\text{d}r)(b\text{d}s)&=a(\text{d}r)(b\text{d}s)&&\\
&=a(b\text{d}s\text{d}r+b\text{d}\{r,s\}+\{r,b\}\text{d}s)&&(\text{by }(\ref{Kahl2}))\\
&=ab(\text{d}s\text{d}r)+ab\text{d}\{r,s\}+a\{r,b\}\text{d}s&&\\
&=ab(\text{d}r\text{d}s+\text{d}\{s,r\})+ab\text{d}\{r,s\}+a\{r,b\}\text{d}s&&(\text{by }(\ref{Kahl2}))\\
&=ab\text{d}r\text{d}s+a\{r,b\}\text{d}s.
\end{aligned}$$
\end{proof}

\begin{lem}\label{LEM11}
The ${\bf k}$-algebra $R\rtimes U(\Omega_{R/{\bf k}})$ is generated by the elements
$$a\otimes1,\ \ \ 1\otimes b\text{d}r\ \ (a,b,r\in R).$$
\end{lem}

\begin{proof}
Note that every element of $R\rtimes U(\Omega_{R/{\bf k}})$ is a ${\bf k}$-linear combination of the elements
$a\otimes f$ for some $a\in R$ and $f\in U(\Omega_{R/{\bf k}})$ and that $f$ is ${\bf k}$-linear combination of  finite products of the form $b\text{d} r$ for some $b,r\in R$. Hence the result is proved easily  from the following multiplicative rules
\begin{equation}\label{LEM111
}\begin{aligned}
&a\otimes f=(a\otimes 1)(1\otimes f),\\
 &(1\otimes b\text{d}r)(1\otimes c\text{d}s)=1\otimes (b\text{d}r)(c\text{d}s),
 \end{aligned}
 \end{equation}
where $a,b,c,r,s\in R$ and $f\in U(\Omega_{R/{\bf k}})$.
\end{proof}

Denote by $R^e$ the quotient algebra
\begin{equation}\label{Kahl4}
R^e:=R\rtimes U(\Omega_{R/{\bf k}})/( a\otimes \text{d}r-1\otimes a\text{d}r\ |\ a,r\in R)
\end{equation}
and let  $i$ and $\text{d}$ be the canonical maps
$$\begin{aligned}
&i:R\longrightarrow R^e,&& i(a)=\overline{a\otimes 1},\\
&\text{d}:R\longrightarrow R^e,&& \text{d}(a)=\overline{1\otimes\text{d}a}.
\end{aligned}$$

\begin{lem}\label{LEM1111}
The canonical maps  $i$ and $\text{d}$ are  an algebra homomorphism and a Lie algebra homomorphism from $(R,\{-,-\})$ into $R^e_L$, respectively, and
the pair $(i,\text{d})$ satisfies the property {\bf P} from $R$ into $R^e$.
\end{lem}

\begin{proof}
It is clear that $i$ and $\text{d}$ are ${\bf k}$-linear maps. For $a,b,r,s\in R$,
$$i(ab)=\overline{ab\otimes 1}=\overline{(a\otimes1)(b\otimes1)}=i(a)i(b)$$
and
$$\begin{aligned}
\ [\text{d}(r), \text{d}(s)]&=\overline{(1\otimes\text{d}r)(1\otimes\text{d}s)-(1\otimes\text{d}s)(1\otimes\text{d}r)}\\
&=\overline{1\otimes[\text{d}r,\text{d}s]}&&(\text{by (\ref{Kahl3})})\\
&=\overline{1\otimes\text{d}\{r,s\}}&&(\text{by (\ref{Kahl2})})\\
&=\text{d}(\{r,s\}).
\end{aligned}
$$
Thus $i$ is an algebra homomorphism and $\text{d}:(R,\{-,-\})\longrightarrow R^e_L$ is a Lie algebra homomorphism.

For $a,r,s\in R$,
$$
\begin{aligned}
\ [\text{d}(r), i(a)]&=\overline{(1\otimes \text{d}r)(a\otimes 1)-(a\otimes1)(1\otimes \text{d}r)}\\
&=\overline{(\text{d}r\cdot a)\otimes 1+(1\cdot a)\otimes\text{d}r-(a\otimes \text{d}r)}&&(\text{by (\ref{Kahl3})})\\
&=\overline{\{r,a\}\otimes 1}&&(\text{by (\ref{ACT})})\\
&=i(\{r,a\})
\end{aligned}$$
 and
$$
\begin{aligned}
\text{d}(rs)&=\overline{1\otimes \text{d}(rs)}\\
&=\overline{1\otimes(r\text{d}s+s\text{d}r)}&&(\text{by (\ref{Kahl})})\\
&=\overline{1\otimes r\text{d}s}+\overline{1\otimes s\text{d}r}\\
&=\overline{r \otimes \text{d}s}+\overline{s\otimes \text{d}r}&&(\text{by (\ref{Kahl4})})\\
&=\overline{(r\otimes1)(1 \otimes \text{d}s)}+\overline{(s\otimes1)(1\otimes \text{d}r)}&&(\text{by (\ref{Kahl3})})\\
&=i(r)\text{d}(s)+i(s)\text{d}(r).
\end{aligned}
$$
Thus the pair $(i,\text{d})$ satisfies the property {\bf P} from $R$ into $R^e$.
\end{proof}

\begin{prop}\label{PENV}
Let $R$ be a Poisson algebra.

(1) The K$\ddot{a}$hler differential   $\Omega_{R/{\bf k}}$ is a left $R$-module as well as a ${\bf k}$-Lie algebra with Lie bracket (\ref{Kahl2}). Denote by $U(\Omega_{R/{\bf k}})$ the universal enveloping algebra of $\Omega_{R/{\bf k}}$.

(2) The Poisson algebra $R$ is a left $U(\Omega_{R/{\bf k}})$-module algebra with action
$$a \text{d}r\cdot b=a\{r,b\}$$
for all $a,b,r\in R$. Hence there exists the semi-crossed product $R\rtimes U(\Omega_{R/{\bf k}})$ with multiplication
(\ref{Kahl3}).

(3) The triple $(R^e, i,\text{d})$ is the Poisson enveloping algebra of $R$, where
$$R^e=R\rtimes U(\Omega_{R/{\bf k}})/(a\otimes \text{d}r-1\otimes a\text{d}r\ |\ a,r\in R)$$
$$\begin{aligned}
&i:R\longrightarrow R^e,&& i(a)=\overline{a\otimes 1}\\
&\text{d}:R\longrightarrow R^e,&& \text{d}(a)=\overline{1\otimes\text{d}a}.
\end{aligned}$$

Note that $i$ is injective by \cite[Proposition 2.2]{OhPaSh2}.
Writing $a$ and $\text{d}r$ for the images $i(a)$ and $\text{d}(r)$ respectively,  $R^e$ is a ${\bf k}$-algebra generated by $R$ and $\text{d}r$ for all $r\in R$  subject to the relations
\begin{equation}\label{Kahl5}
\begin{aligned}
\ [\text{d}r,a]&=\{r,a\}\\
[\text{d}r,\text{d}s]&=\text{d}\{r,s\}\\
\end{aligned}
\end{equation}
for $a,r,s\in R$.

(4) Let $S$ be a Poisson subalgebra of $R$. Then the Poisson enveloping algebra of $S$ is $$S^e=(S^e, i|_S,\text{d}|_S),$$ where $S^e$ is  the subalgebra of $R^e$ generated by $S$ and $\text{d}s$ for all $s\in S$ and
 $i|_S$ and $\text{d}|_S$ are the restrictions of $i$ and $\text{d}$ respectively.
\end{prop}

\begin{proof}
(1) and  (2) are proved already.

(3) By Lemma~\ref{LEM1111}, the pair $(i,\text{d})$  satisfies the property {\bf P} from $R$ into $R^e$.
Let $A$ be an algebra and let $(\gamma,\delta)$ satisfy the property {\bf P} from $R$ into $A$.
Define a ${\bf k}$-linear map $h'$ from $\Omega_{R/{\bf k}}$ into $A$ by
$$h'(b\text{d}r)=\gamma(b)\delta(r)$$
for all $b,r\in R$. Since $\gamma$ is an algebra homomorphism and $\delta$ is a Lie algebra homomorphism, $h'$ satisfies the relations (\ref{Kahl}) and thus $h'$ is well defined. Moreover, for $b,c,r,s\in R$,
$$\begin{aligned}
h'([b\text{d}r, c\text{d}s])&=h'(bc\text{d}\{r,s\}+b\{r,c\}\text{d}s-c\{s,b\}\text{d}r)\ \ \ \ \ \ \ \ \ \ \ \ \ \  \ (\text{by (\ref{Kahl2})})\\
&=\gamma(bc)\delta(\{r,s\})+\gamma(b\{r,c\})\delta(s)-\gamma(c\{s,b\})\delta(r)\\
&=\gamma(bc)\delta(\{r,s\})+\gamma(b)\gamma(\{r,c\})\delta(s)-\gamma(c)\gamma(\{s,b\})\delta(r)\\
&=\gamma(bc)(\delta(r)\delta(s)-\delta(s)\delta(r))+\gamma(b)[\delta(r),\gamma(c)]\delta(s)-\gamma(c)[\delta(s),\gamma(b)]\delta(r)\\
& \hspace{3in}  (\text{by the property {\bf P}})\\
&=\gamma(b)\delta(r)\gamma(c)\delta(s)-\gamma(c)\delta(s)\gamma(b)\delta(r)\\
&=[\gamma(b)\delta(r),\gamma(c)\delta(s)]\\
&=[h'(b\text{d}r),h'(c\text{d}s)]
\end{aligned}$$
and thus $h'$ is a Lie algebra homomorphism from $\Omega_{R/{\bf k}}$ into $A_L$. It follows that $h'$ is extended to
$U(\Omega_{R/{\bf k}})$.

Define a ${\bf k}$-linear map  $h: R\rtimes U(\Omega_{R/{\bf k}})\longrightarrow A$ by $h(a\otimes f)=\gamma(a)h'(f)$.
Thus
\begin{equation}\label{Kahl6}
h(1\otimes b\text{d}r)=\gamma(1)h'(b\text{d}r)=\gamma(b)\delta(r),\ \ h(a\otimes1)=\gamma(a)h'(1)=\gamma(a)
\end{equation}
for $a,b,r\in R$. Note that  $R\rtimes U(\Omega_{R/{\bf k}})$ is generated by the elements of the form
$1\otimes b\text{d}r$ and $a\otimes1$  by Lemma~\ref{LEM11}. It is checked routinely that
$$\begin{aligned}
h((a\otimes 1)(b\otimes 1))&=h(a\otimes 1)h(b\otimes 1), \\
h((a\otimes1)(1\otimes b\text{d}r))&=h(a\otimes1)h(1\otimes b\text{d}r),\\
h((1\otimes b\text{d}r)(a\otimes1))&=h(1\otimes b\text{d}r)h(a\otimes1), \\
h((1\otimes a\text{d}s)(1\otimes b\text{d}r))&=h(1\otimes a\text{d}s)h(1\otimes b\text{d}r)
\end{aligned}$$
for all $a,b,r,s\in R$ and thus $h$ is an algebra homomorphism.
 Since
$$h(a\otimes \text{d}r)=\gamma(a)\delta(r)=h(1\otimes a\text{d}r)$$
for all $a,r\in R$, there exists the  algebra homomorphism $\overline{h}:R^e\longrightarrow A$  induced by $h$.
Since $\overline{h}i=\gamma$ and $\overline{h}\text{d}=\delta$ by (\ref{Kahl6}) and $R^e$ is generated by the images of $i$ and $\text{d}$, $\overline{h}$ is determined uniquely.
Hence $(R^e, i,\text{d})$ is a Poisson enveloping algebra of $R$.

By Lemma~\ref{LEM11} and (\ref{Kahl4}), $R^e$ is generated by $R$ and $\text{d} r$ for all $r\in R$.
The relations (\ref{Kahl5}) are already shown in the proof of Lemma~\ref{LEM1111}. Thus the remaining assertion holds.

(4) The restriction $\text{d}|_S$ satisfies (\ref{Kahl}), the pair $(i|_S,\text{d}|_S)$ satisfies the property ${\bf P}$ and $S^e$ is a ${\bf k}$-algebra generated by $S$ and $\text{d} s$ for $s\in S$  subject to the relations (\ref{Kahl5}). Hence, replacing $R$ by $S$ in the second statement of (3),  $S^e$ is the Poisson enveloping algebra of $S$.
\end{proof}

\section{Poisson enveloping algebra of double Poisson-Ore extension}

Let us recall a left double Ore extension, shortly a left double extension, of an algebra $R$ defined in \cite[\S1]{ZhZh}. (In which it is called a {\it right} double extension.) Let $\Bbb F$ be a commutative ${\bf k}$-algebra and let $R$ be an $\Bbb F$-algebra. An $\Bbb F$-algebra $A$ containing $R$ as a subalgebra is said to be a {\it left double extension} of $R$ if
$A$ is generated by $R$ and new variables $y_1,y_2$ such that
\begin{itemize}
\item $y_1$ and $y_2$ satisfy a relation
\begin{equation}\label{REL1}
y_2y_1=p_{11}y_1^2+p_{12}y_1y_2+\tau_1y_1+\tau_2y_2+\tau_0,
\end{equation}
where $P:=(p_{11},p_{12})\in\Bbb F^2$ and $\tau:=(\tau_1,\tau_2, \tau_0)\in R^3$,
\item As a left $R$-module, $A$ is a free left $R$-module with a basis $\{y_1^iy_2^j|i,j\geq0\}$,
\item $y_1R+y_2R+R\subseteq Ry_1+Ry_2+R$.
\end{itemize}
Hence there exist $\Bbb F$-linear maps $\sigma_{11},\sigma_{12},\sigma_{21},\sigma_{22},\delta_1,\delta_2$ from $R$ into itself such that
\begin{equation}\label{REL2}
\begin{aligned}
y_1a&=\sigma_{11}(a)y_1+\sigma_{12}(a)y_2+\delta_1(a),\\
y_2a&=\sigma_{21}(a)y_1+\sigma_{22}(a)y_2+\delta_2(a)
\end{aligned}
\end{equation}
for all $a\in R$.  Set
$$\begin{aligned}
&y:=\left(\begin{matrix} y_1\\ y_2\end{matrix}\right)\in M_{2\times1}(A),&&\\
&\sigma:R\longrightarrow M_{2\times2}(R),&\ \ \ \sigma(a)&=\left(\begin{matrix}\sigma_{11}(a)&\sigma_{12}(a)\\ \sigma_{21}(a)&\sigma_{22}(a)\end{matrix}\right),\\
&\delta:R\longrightarrow M_{2\times1}(R),&\ \ \ \delta(a)&=\left(\begin{matrix}\delta_{1}(a)\\ \delta_{2}(a)\end{matrix}\right).
\end{aligned}$$
Note that $ M_{2\times1}(A)$, $M_{2\times2}(R)$ and $M_{2\times1}(R)$ are both  left and  right $R$-modules and that (\ref{REL2}) is expressed explicitly by
$$ya=\sigma(a)y+\delta(a)$$
for all $a\in R$. We say that the left double extension $A$ of $R$ has the DE-data $\{P,\sigma,\delta,\tau\}$ and $A$ is denoted by
$$A=R[y_1,y_2;\sigma,\delta].$$

By symmetry, we have the notion of  {\it right} double Ore extension, shortly a right double extension. An algebra $A$ is said to be a {\it double Ore extension} of $R$, shortly a double extension, if it is a left and a right double  extension of $R$ with same generating set.

In \cite[Theorem 2.7]{LoOhWa}, a double Poisson-Ore extension is defined as the semiclassical limit of a left double extension as follows.
 Let $R$ be a Poisson ${\bf k}$-algebra with Poisson bracket $\{-,-\}_R$ and let $R[y_1,y_2]$ be the commutative polynomial ring. Set
 $$\begin{aligned}
&Q=\{q_{11},q_{12}\}\subset {\bf k},&w&=\{w_1,w_2, w_0\}\subset R,\\
&\alpha:R\longrightarrow M_{2\times2}(R),& \alpha(a)&=\left(\begin{matrix}\alpha_{11}(a)&\alpha_{12}(a)\\ \alpha_{21}(a)&\alpha_{22}(a)\end{matrix}\right),\\
&\nu:R\longrightarrow M_{2\times1}(R),& \nu(a)&=\left(\begin{matrix}\nu_{1}(a)\\ \nu_{2}(a)\end{matrix}\right),\\
&y=\left(\begin{matrix} y_1\\ y_2\end{matrix}\right)\in M_{2\times1}(R[y_1,y_2]).
\end{aligned}$$
Note that $ M_{2\times1}(R[y_1,y_2])$, $M_{2\times2}(R)$ and $M_{2\times1}(R)$ are  Poisson $R$-modules.
 Then $R[y_1,y_2]$  becomes a Poisson algebra with Poisson bracket
\begin{equation}\label{PB11}
\begin{aligned}
\{a,b\}&=\{a,b\}_R,\\
\{y_2,y_1\}&=q_{11}y_1^2+q_{12}y_1y_2+w_1y_1+w_2y_2+w_0,\\
\{y,a\}&=\alpha(a)y+\nu(a)
\end{aligned}
\end{equation}
for all $a,b\in R$ if and only if the DE-data $\{Q,\alpha,\nu,w\}$ satisfies the following conditions (a)-(e).
\begin{enumerate}
\item[(a)]  $\alpha(ab)=a\alpha(b)+b\alpha(a)$.
\item[(b)] $\nu(ab)=a\nu(b)+b\nu(a)$.
\item[(c)] $\alpha(\{a,b\})=\{\alpha(a),b\}+\{a,\alpha(b)\}+[\alpha(a),\alpha(b)]$.
\item[(d)] $\nu(\{a,b\})=\{\nu(a),b\}+\{a,\nu(b)\}+\alpha(a)\nu(b)-\alpha(b)\nu(a)$.
\item[(e)] $\{y_2,\{y_1,a\}\}+\{y_1,\{a,y_2\}\}+\{a,\{y_2,y_1\}\}=0$.
\end{enumerate}

The Poisson algebra $R[y_1,y_2]$ with Poisson bracket (\ref{PB11}) is called a {\it double Poisson-Ore extension} with DE-data $\{Q,\alpha,\nu,w\}$ and denoted by
$$R[y_1,y_2;\alpha,\nu]_p.$$

\begin{thm}\label{ENVDPE}
Let $R$ be a  Poisson algebra and let $A=R[x_1,x_2;\alpha,\nu]_p$ be a double Poisson-Ore extension of $R$ with DE-data
$$\left\{Q=(q_{11},q_{12}),\ \alpha=\begin{pmatrix}\alpha_{11}&\alpha_{12}\\ \alpha_{21}&\alpha_{22}  \end{pmatrix},\ \nu=\begin{pmatrix}\nu_{1}\\ \nu_{2} \end{pmatrix},\ w=(w_1,w_2,w_0)\right\}.$$
Then the Poisson enveloping algebra $A^e$  is an iterated double extension
$$R^e[x_1,x_2;\sigma^1,\delta^1][y_1,y_2;\sigma^2,\delta^2]$$
over the Poisson enveloping algebra $R^e$.
Where the DE-data  $\{P^1,\sigma^1,\delta^1,\tau^1\}$ of $R^e[x_1,x_2;\sigma^1,\delta^1]$ is
$$\begin{aligned}
P^1&=(0,1), &\tau^1&=(0,0,0),\\
\sigma^1(a)&=\begin{pmatrix}a&0\\ 0&a  \end{pmatrix},&
\sigma^1(\text{d}a)&=\begin{pmatrix}\alpha_{11}(a)+\text{d}a&\alpha_{12}(a)\\ \alpha_{21}(a)&\alpha_{22}(a)+\text{d}a  \end{pmatrix},\\
\delta^1(a)&=\begin{pmatrix}0\\ 0  \end{pmatrix},&
\delta^1(\text{d}a)&=\begin{pmatrix}\nu_{1}(a)\\ \nu_{2}(a) \end{pmatrix}
\end{aligned}$$
 for all $a\in R$
 and the DE-data $\{P^2,\sigma^2,\delta^2,\tau^2\}$ of $(R^e[x_1,x_2;\sigma^1,\delta^1])[y_1,y_2;\sigma^2,\delta^2]$ is
$$\begin{aligned}
P^2&=(0,1), \\
\tau^2&=(2q_{11}x_1+q_{12}x_2+w_1, \ q_{12}x_1+w_2, \ x_1\text{d}w_1+x_2\text{d}w_2+\text{d}w_0),\\
\sigma^2(a)&=\begin{pmatrix}a&0\\ 0&a  \end{pmatrix},\ \ \ \ \ \
\sigma^2(\text{d}a)=\begin{pmatrix}\alpha_{11}(a)+\text{d}a&\alpha_{12}(a)\\ \alpha_{21}(a)&\alpha_{22}(a)+\text{d}a  \end{pmatrix},\\
\sigma^2(x_1)&=\begin{pmatrix}x_1&0\\ 0&x_1  \end{pmatrix}, \ \ \
\sigma^2(x_2)=\begin{pmatrix}x_2&0\\ 0&x_2  \end{pmatrix}, \\
\delta^2(a)&=\begin{pmatrix}\alpha_{11}(a)x_1+\alpha_{12}(a)x_2+\nu_1(a)\\ \alpha_{21}(a)x_1+\alpha_{22}(a)x_2+\nu_2(a)  \end{pmatrix},&&\\
\delta^2(\text{d}a)&=\begin{pmatrix}x_1\text{d}\alpha_{11}(a)+x_2\text{d}\alpha_{12}(a) + \text{d}\nu_1(a)\\
x_1\text{d}\alpha_{21}(a)+x_2\text{d}\alpha_{22}(a) + \text{d}\nu_2(a) \end{pmatrix}, \\
\delta^2(x_1)&=\begin{pmatrix}0\\ q_{11}x_1^2+q_{12}x_1x_2+w_1x_1+w_2x_2+w_0 \end{pmatrix},\\
\delta^2(x_2)&=\begin{pmatrix}-(q_{11}x_1^2+q_{12}x_1x_2+w_1x_1+w_2x_2+w_0)\\ 0 \end{pmatrix}
\end{aligned}$$
for all $a\in R$.
\end{thm}

\begin{proof}
In the  K$\ddot{a}$hler differential $\Omega_{A/{\bf k}}$, set
$$y_1:=\text{d}x_1, \ \ y_2:=\text{d}x_2.$$
By Proposition~\ref{PENV}, $A^e$ is a ${\bf k}$-algebra generated by
$$R, x_1, x_2, \text{d}a, y_1, y_2, \ \ (a\in R)$$
with the following relations: for any $a,b\in R$ and $k=1,2$,
\begin{equation}\label{RELA1}
\begin{aligned}
& [\text{d}a,b]=\{a,b\},&&\\
&[\text{d}a,\text{d}b]=\text{d}\{a,b\},&&\\
\end{aligned}\end{equation}
\begin{equation}\label{RELA2}
\begin{aligned}
&[x_k, a]=0,&& \\
&[x_k, \text{d}a]=\{x_k,a\}=\alpha_{k1}(a)x_1+\alpha_{k2}(a)x_2+\nu_k(a),&& \\
&[x_2,x_1]=0,&&\\
\end{aligned}\end{equation}
\begin{equation}\label{RELA3}
\begin{aligned}
\ [y_k,a]&=\{x_k,a\}=\alpha_{k1}(a)x_1+\alpha_{k2}(a)x_2+\nu_k(a),&&\\
[y_k, \text{d}a]&=\text{d}\{x_k,a\}=\text{d}(\alpha_{k1}(a)x_1+\alpha_{k2}(a)x_2+\nu_k(a)),&& \\
&=(\alpha_{k1}(a)y_1+\alpha_{k2}(a)y_2) +(x_1\text{d}\alpha_{k1}(a)+x_2\text{d}\alpha_{k2}(a) + \text{d}\nu_k(a)),&&\\
[y_k,x_k]&=0,&&\\
[y_1,x_2]&=\{x_1,x_2\}=-(q_{11}x_1^2+q_{12}x_1x_2+w_1x_1+w_2x_2+w_0),&&\\
[y_2,x_1]&=\{x_2,x_1\}=q_{11}x_1^2+q_{12}x_1x_2+w_1x_1+w_2x_2+w_0,&&\\
[y_2,y_1]&=\text{d}\{x_2,x_1\}=\text{d}(q_{11}x_1^2+q_{12}x_1x_2+w_1x_1+w_2x_2+w_0)&&\\
&=(2q_{11}x_1+q_{12}x_2+w_1)y_1+(q_{12}x_1+w_2)y_2&&\\
&\qquad\qquad\qquad\qquad\qquad +(x_1\text{d}w_1+x_2\text{d}w_2+\text{d}w_0).&&
\end{aligned}\end{equation}

Note that   $R^e$  is the subalgebra of $A^e$ generated by $R$ and $\text{d}a$ for all $a\in R$ by Proposition~\ref{PENV}(4). Let
 $B$ be the subalgebra of $A^e$ generated by $R^e$ and $x_1,x_2$.

Let $Z$ be a generating set of $R$ as an algebra. The K$\ddot{a}$hler differential $\Omega_{R/{\bf k}}$ is a left $R$-module generated by  $\text{d}R:=\{\text{d} r | r\in R\}$ and every element $\text{d}r\in\text{d}R$ is an $R$-linear combination of $\{\text{d}z | z\in Z\}$.
Hence $\Omega_{R/{\bf k}}$ is generated by $\{\text{d}z | z\in Z\}$ as a left $R$-module.
 Let $X$ be a maximal ${\bf k}$-linearly independent subset of $\{\text{d}z | z\in Z\}$. Note that $X$ is a generating set of $\Omega_{R/{\bf k}}$ as a left $R$-module. Set
$$X=\{\text{d} z_j | j\in J\}$$
and give a well-order relation $\leq$ on $J$. Let $G_1$ be the semigroup $\bigoplus_{j\in J} S_j$, where each $S_j$ is the semigroup $\Bbb Z_+:=\{0,1,2,\ldots\}$ with the usual addition and let $G_2,G_3$ be the semigroup $\Bbb Z_+\times \Bbb Z_+$.
Give an order relation $\leq$ in $G_1$ as follows: Let $e_j$ be the canonical element of $G_1$ such that the $j$-th component is 1 and the others are 0. For $\sum_{j\in J} p_je_j,\sum_{j\in J}q_je_j\in G_1$,
$$\begin{aligned}
\sum_{j\in J} p_je_j&< \sum_{j\in J}q_je_j
\Leftrightarrow\\
&\left\{\begin{aligned}&\sum_jp_j<\sum_jq_j\text{ or }\\
&\sum_jp_j=\sum_jq_j\text{ and }\exists\, j_0\in J\ \text{such that}\ p_{j_0}<q_{j_0},  p_j=q_j\, \forall j>j_0.\end{aligned}\right.
\end{aligned}$$
Also, give  order relations $\leq$ in $G_2$ and $G_3$ as follows:  For $(m,n),(p,q)\in G_2, G_3$,
$$(m,n)< (p,q)\Leftrightarrow
\left\{\begin{aligned}&m+n<p+q\text{ or }\\
&m+n=p+q\text{ and } n< q\text{ or }\\
&m+n=p+q,  n= q\text{ and }m<p.\end{aligned}\right.$$
Set $$G=G_1\times G_2\times G_3$$
and give an order relation $\preccurlyeq$ on $G$ as follows: For any $(a_1,a_2, a_3), (b_1, b_2, b_3)\in G$,
$$ (a_1,a_2, a_3)\preccurlyeq (b_1, b_2, b_3) \Leftrightarrow  \left\{\begin{aligned} & a_3< b_3 \text{ or }\\
                                                      & a_3= b_3 \text{ and  } a_2< b_2 \text{ or } \\
                                                       & a_3= b_3 , a_2=b_2 \text{ and  } a_1\leq b_1.
                                                       \end{aligned}\right.$$
We will identify $G_1, G_2, G_3$ with the corresponding canonical sub-semigroups of $G$. Note that  the order relation $\preccurlyeq$ on $G$ is the reversed lexicographic order and
$$g_1\prec g_2\prec g_3$$
for any nonzero elements $g_1\in G_1, g_2\in G_2, g_3\in G_3$.
          We will call finite products of $a\in R, \text{d}z_j, x_1, x_2,y_1,y_2$  {\it monomials}, where $j\in J$ and repetitions allowed. A monomial ${\bf x}$ is said to be a {\it standard monomial} if ${\bf x}$ is of the form
$${\bf x}=a(\text{d}z_{j_1})(\text{d}z_{j_2})\ldots(\text{d}z_{j_k})x_1^mx_2^ny_1^py_2^q,$$
where $a\in R$, $j_i\in J$, $j_1\leq j_2\leq\ldots\leq j_k$ and $m,n,p,q\in\Bbb Z_+$.
Note that every element of $A^e$ is a  ${\bf k}$-linear combination of  monomials. Give  degrees on the generators of $A^e$ by
\begin{equation}\label{DEG}
\begin{array}{ll}\deg a=0, &(a\in R), \\
\deg \text{d}z_j= e_j\in G_1, &(j\in J),\\
\deg x_1= (1,0)\in G_2, &\deg x_2= (0,1)\in G_2,\\
 \deg y_1= (1,0)\in G_3, &\deg y_2= (0,1)\in G_3.
 \end{array}\end{equation}
 Then every monomial of $A^e$ has a degree induced by (\ref{DEG}). For instance,
the monomial $y_2 a^2 x_1 (\text{d} z_j)^3$ has the degree
$$\deg y_2+\deg a^2+\deg x_1+\deg (\text{d} z_j)^3= (3e_j,(1,0),(0,1))\in G,$$
where $a\in R$.

For  $g\in G$, let $\mathcal{F}_g$ be the ${\bf k}$-linear combinations of monomials with degree less than or equal to $g$. Then, for all $f,g \in G$,
$$\mathcal{F}_f\mathcal{F}_g\subseteq \mathcal{F}_{f+g},\ \ \mathcal{F}_f\subseteq\mathcal{F}_g \ \text{if $f\preccurlyeq g$}, \ \ \
\bigcup_{g\in G}\mathcal{F}_g = A^e. $$
Hence, $\mathcal{F}(A^e):=\{\mathcal{F}_g\  |\ g\in G\}$ is a filtration of $A^e$. Observe that $\mathcal{F}_0=R$, where $0$ is the identity element of $G$, and that
$$\mathcal{F}(R^e):=\{\mathcal{F}_g\cap R^e\  |\ g\in G\},\ \ \ \mathcal{F}(B):=\{\mathcal{F}_g\cap B \ |\ g\in G\}$$
are also filtrations of $R^e$ and $B$, respectively. Let $\text{gr}_{\mathcal F}(A^e)$ be the associated graded algebra determined by $\mathcal{F}(A^e)$. That is,
$$\text{gr}_{\mathcal F}(A^e)=\bigoplus_{g\in G} (\mathcal{F}_g/\mathcal{F}_{g^-}),$$
where $\mathcal{F}_{g^-}$ is the ${\bf k}$-linear combinations of monomials with degree strictly less than $g$ ($\mathcal{F}_{0^-}=\{0\}$).  Refer to \cite[\S1.6]{McRo} for details of the associated graded algebra.
The associated graded algebras $\text{gr}_{\mathcal F}(R^e)$ and $\text{gr}_{\mathcal F}(B)$ are also constructed by the filtrations $\mathcal{F}(R^e)$ and $\mathcal{F}(B)$, respectively.

\begin{lem}\label{GRAD}
(1) $\text{gr}_{\mathcal F}(R^e)$ and $\text{gr}_{\mathcal F}(B)$
are subalgebras of $\text{gr}_{\mathcal F}(B)$ and $\text{gr}_{\mathcal F}(A^e)$, respectively.

(2) $\text{gr}_{\mathcal F}(R^e)$ is a commutative algebra.

(3) $\text{gr}_{\mathcal F}(B)$ is a polynomial algebra over $\text{gr}_{\mathcal F}(R^e)$ with two variables
 $$\text{gr}_{\mathcal F}(B)=\text{gr}_{\mathcal F}(R^e)[\overline{x}_1,\overline{x}_2],$$
 where $\overline{x}_1, \overline{x}_2$ are the canonical images of $x_1,x_2$ in $\text{gr}_{\mathcal F}(B)$, respectively.

(4) $\text{gr}_{\mathcal F}(A^e)$ is a polynomial  algebra over $\text{gr}_{\mathcal F}(B)$ with two variables
$$\text{gr}_{\mathcal F}(A^e)=\text{gr}_{\mathcal F}(B)[\overline{y}_1,\overline{y}_2]
=\text{gr}_{\mathcal F}(R^e)[\overline{x}_1,\overline{x}_2][\overline{y}_1,\overline{y}_2],$$
where $\overline{y}_1, \overline{y}_2$ are the canonical images of $y_1,y_2$ in $\text{gr}_{\mathcal F}(A^e)$, respectively.

(5) Every element of $A^e$ (respectively, $B$, $R^e$) is a ${\bf k}$-linear combination of standard monomials.

(6) For any nonzero element $z\in A^e$, there exists $g\in G$ such that
$$0\neq z+\mathcal{F}_{g^-}\in \mathcal{F}_{g}/\mathcal{F}_{g^-} \subseteq\text{gr}_{\mathcal F}(A^e).$$
\end{lem}

\begin{proof}
(1) It is obvious since $(\mathcal{F}_g\cap R^e)\subseteq (\mathcal{F}_g\cap B)\subseteq \mathcal{F}_g$ for each $g\in G$.

(2)
In the commutation relations (\ref{RELA1}),   the degrees of monomials appearing in the left hand sides  are greater than those of monomials appearing in the right hand sides. Hence $\text{gr}_{\mathcal F}(R^e)$ is commutative.

(3)   The result follows immediately from (\ref{RELA2}).

(4) The result follows immediately from  (\ref{RELA3}).

(5) It is obvious by (\ref{RELA1}),  (\ref{RELA2}) and (\ref{RELA3}).

(6) Let $g=\min\{f\in G\ |\ z\in \mathcal{F}_f\}$. Then $0\neq z+\mathcal{F}_{g^-}\in \mathcal{F}_{g}/\mathcal{F}_{g^-} \subseteq\text{gr}_{\mathcal F}(A^e).$
\end{proof}

By Lemma~\ref{GRAD}(3),  $B$ is generated by $$\frak B=\{x_1^{\ell_1}x_2^{\ell_2}\, |\, \ell_1,\ell_2\geq0\}$$ as a left $R^e$-module. Suppose that
$$\sum_{k,\ell} z_{k,\ell}x_1^{k}x_2^{\ell}=0,$$
where $z_{k,\ell}\in R^e$ for all $k,\ell$. Since  $\text{gr}_{\mathcal F}(B)$ is the polynomial ring $\text{gr}_{\mathcal F}(R^e)[\overline{x}_1,\overline{x}_2]$, the corresponding elements
of $z_{k,\ell}$ in $\text{gr}(\mathcal{F}(R^e))$ are zero for all $k,\ell$. Hence $z_{k,\ell}=0$ for all $k,\ell$ by Lemma~\ref{GRAD}(6) and thus $B$ is a free left $R^e$-module with basis $\frak B$.
It follows, by (\ref{RELA2}), that
$B$ is a left double  extension
$$R^e[x_1,x_2;\sigma^1,\delta^1]$$
with the DE-data $\{P^1,\sigma^1,\delta^1,\tau^1\}$ given by
$$\begin{aligned}
P^1&=(0,1), &\tau^1&=(0,0,0),\\
\sigma^1(a)&=\begin{pmatrix}a&0\\ 0&a  \end{pmatrix},&
\sigma^1(\text{d}a)&=\begin{pmatrix}\alpha_{11}(a)+\text{d}a&\alpha_{12}(a)\\ \alpha_{21}(a)&\alpha_{22}(a)+\text{d}a  \end{pmatrix},\\
\delta^1(a)&=\begin{pmatrix}0\\ 0  \end{pmatrix},&
\delta^1(\text{d}a)&=\begin{pmatrix}\nu_{1}(a)\\ \nu_{2}(a) \end{pmatrix}
\end{aligned}$$
for $a\in R$.  Moreover, $B$ is a free right $R^e$-module with basis $\frak B$ by (\ref{RELA2}) and thus $B$ is a double extension of $R^e$ since $x_1x_2=x_2x_1$.

We have already known that $A^e$ is generated by
$$\frak C=\{y_1^{\ell_1}y_2^{\ell_2}\, |\, \ell_1,\ell_2\geq0\}$$ as a left $B$-module.
Since $\text{gr}_{\mathcal F}(A^e))$ is the  polynomial ring $\text{gr}_{\mathcal F}(B)[\overline{y}_1,\overline{y}_2]$, $A^e$ is a free left $B$-module with basis $\frak C$ by Lemma~\ref{GRAD}(4).
 Hence,  by the commutation relations (\ref{RELA3}),  $A^e$ is a left double  extension $B[y_1,y_2;\sigma^2,\delta^2]$ with the DE-data $\{P^2,\sigma^2,\delta^2,\tau^2\}$, where
$$\begin{aligned}
P^2&=(0,1), \\
\tau^2&=(2q_{11}x_1+q_{12}x_2+w_1, \ q_{12}x_1+w_2, \ x_1\text{d}w_1+x_2\text{d}w_2+\text{d}w_0),\\
\sigma^2(a)&=\begin{pmatrix}a&0\\ 0&a  \end{pmatrix},\ \ \ \ \ \
\sigma^2(\text{d}a)=\begin{pmatrix}\alpha_{11}(a)+\text{d}a&\alpha_{12}(a)\\ \alpha_{21}(a)&\alpha_{22}(a)+\text{d}a  \end{pmatrix},\\
\sigma^2(x_1)&=\begin{pmatrix}x_1&0\\ 0&x_1  \end{pmatrix}, \ \ \
\sigma^2(x_2)=\begin{pmatrix}x_2&0\\ 0&x_2  \end{pmatrix}, \\
\delta^2(a)&=\begin{pmatrix}\alpha_{11}(a)x_1+\alpha_{12}(a)x_2+\nu_1(a)\\ \alpha_{21}(a)x_1+\alpha_{22}(a)x_2+\nu_2(a)  \end{pmatrix},&&\\
\delta^2(\text{d}a)&=\begin{pmatrix}x_1\text{d}\alpha_{11}(a)+x_2\text{d}\alpha_{12}(a) + \text{d}\nu_1(a)\\
x_1\text{d}\alpha_{21}(a)+x_2\text{d}\alpha_{22}(a) + \text{d}\nu_2(a) \end{pmatrix}, \\
\delta^2(x_1)&=\begin{pmatrix}0\\ q_{11}x_1^2+q_{12}x_1x_2+w_1x_1+w_2x_2+w_0 \end{pmatrix},\\
\delta^2(x_2)&=\begin{pmatrix}-(q_{11}x_1^2+q_{12}x_1x_2+w_1x_1+w_2x_2+w_0)\\ 0 \end{pmatrix}.
\end{aligned}$$
Moreover, $A^e$ is  a free right $B$-module with basis $\{y_2^iy_1^j\}_{i,j\geq0}$  by (\ref{RELA3}). Hence $A^e$ is also a right double  extension of $B$ and thus it is a double  extension of $B$. It completes the proof of Theorem~\ref{ENVDPE}.
\end{proof}

\begin{notn}\label{RemarkGR} Let $R$ and $G_1$ be the ones in the proof of  Theorem~\ref{ENVDPE}.

(1) The filtration of $R^e$ is indexed by the semigroup $G_1$, namely,
$$\mathcal{F}(R^e)=\{\mathcal{F}_{g}(R^e)\,|\,g\in G_1\},$$
where $\mathcal{F}_g(R^e)=\mathcal{F}_{g}\cap R^e$, and its associated graded algebra $\text{gr}_{\mathcal F}(R^e)$ is a commutative algebra generated by
$R $ and $\overline{\text{d}z}_j$ for all $j\in J$. Hence, if $R$ is finitely generated then $\text{gr}_{\mathcal F}(R^e)$
is a finitely generated commutative algebra over $R$. It follows that the Poisson enveloping algebra of any Poisson algebra that is finitely generated as an algebra is noetherian. (See \cite[Proposition 9]{Oh5}.)

(2)
Let $A$ be a Poisson-Ore extension $A=R[x;\alpha,\nu]_p$ given in \cite{Oh8}, namely, $A=R[x]$ is a Poisson algebra with a Poisson bracket
$$\{x,a\}=\alpha(a)x+\nu(a)$$
for $a\in R$.
 Set $$y=\text{d}x\in A^e,\ \ \ G_2=G_3=\Bbb Z_+,\ \ \ G=G_1\times G_2\times G_3$$
 and give  a well-order relation on $G$ by modifying that of $G$  in the proof of Theorem~\ref{ENVDPE}. If we give degrees on the generators of $A^e$ by
$$
\begin{array}{ll}\deg a=0, &(a\in R), \\
\deg \text{d}z_j= e_j\in G_1, &(j\in J),\\
\deg x= 1\in G_2, & \\
\deg y= 1\in G_3,&
 \end{array}
 $$
 $A^e$ is a filtered algebra with a filtration $\{\mathcal{F}_g| g\in G\}$ induced by the above degrees and
  its associated graded algebra $\text{gr}_{\mathcal F}(A^e)$ is a polynomial ring with two variables
$$\text{gr}_{\mathcal F}(A^e)=\text{gr}_{\mathcal F}(R^e)[\overline{x}][\overline{y}].$$
Hence, the subalgebra $B$ of $A^e$ generated by $R^e$ and $x$ is a free left and right $R^e$-module with basis $\{x^i\ |\ i\geq0\}$
and $A^e$ is a free left and right $B$-module with basis $\{y^i\ |\ i\geq0\}$.
It follows that  $A^e$ is an iterated skew polynomial algebra
$$A^e=R^e[x;\sigma_1,\delta_1][y;\sigma_2,\delta_2],$$
 where $\sigma_k,\delta_k$ $(k=1,2)$ are given by
 \begin{equation}\label{ORE}
 \begin{aligned}
 \sigma_1(a)&=a,&\delta_1(a)&=0, \\
 \sigma_1(\text{d}a)&=\text{d}a+\alpha(a),&\delta_1(\text{d}a)&=\nu(a), \\
 \sigma_2(a)&=a+\alpha(a),&\delta_2(a)&=\nu(a), \\ \sigma_2(\text{d}a)&=\text{d}a+\alpha(a),&\delta_2(\text{d}a)&=x\text{d}a+\text{d}\nu(a),\\
 \sigma_2(x)&=x,&\delta_2(x)&=0
 \end{aligned}
 \end{equation}
 for all $a\in R$. It is easy to observe that $A^e$ is a double extension of $R^e$ with the DE-data determined by (\ref{ORE}).
 (See \cite[Theorem 0.1 and Proposition 2.2]{LuWaZh2}.)
\end{notn}

Rather than Ore extension, very few properties are known to be preserved under double Ore extension. See references \cite{CLM, ZhZh, Zhu}. Hence we do not have an analogy of \cite[Corollary 0.2]{LuWaZh2} saying that the Poisson enveloping algebras of double Poisson-Ore extensions preserve nice properties from the original Poisson enveloping algebras. In the following, we will focus on three special situations where we know the analogy holds.

\begin{cor}\label{IteratedORE}
Let $R$ be a Poisson algebra and let $A=R[x_1,x_2;\alpha,\nu]_p$ be a double Poisson-Ore extension of $R$ with DE-data
$$\left\{Q=(q_{11},q_{12}),\ \alpha=\begin{pmatrix}\alpha_{11}&\alpha_{12}\\ \alpha_{21}&\alpha_{22}  \end{pmatrix},\ \nu=\begin{pmatrix}\nu_{1}\\ \nu_{2} \end{pmatrix},\ w=(w_1,w_2,w_0)\right\}.$$
Suppose $\alpha_{12}=0$ or $\alpha_{21}=0$. Then $A$ is an iterated Poisson-Ore extension of $R$. As a consequence, $A^e$ is an iterated Ore extension of $R^e$, and $A^e$ inherits the following properties from $R^e$:
\begin{itemize}
\item[(1)] being a domain;
\item[(2)] being noetherian;
\item[(3)] having finite global dimension;
\item[(4)] having finite Krull dimension;
\item[(5)] being twisted Calabi-Yau;
\item[(6)] being Koszul provided that $R^e$ and $A^e$ are graded quadratic.
\end{itemize}
\end{cor}
\begin{proof}
Let us assume that $\alpha_{12}=0$. The argument for $\alpha_{21}=0$ is analogous. It is straightforward for one to check that
\[A=R[x_1; \alpha_{11}, \nu_1]_p[x_2;\alpha_{22}', \nu_2']_p\]
is an iterated Poisson-Ore extension of $R$, where $\alpha_{22}'(a)=\alpha_{22}(a)$, $\alpha_{22}'(y_1)=q_{12}y_1+w_2$ and $\nu_2'(a)=\nu_2(a)+\alpha_{21}(a)y_1$, $\nu_2'(y_1)=q_{11}y_1^2+w_1y_1+w_0$ for all $a\in R$. Thus the results follow from \cite[Theorem 0.1\&Corollary 0.2]{LuWaZh2}.
\end{proof}

Let us consider the noetherianess of a left double Ore extension in regard to Corollary~\ref{IteratedORE}(2). It is well-known that an  Ore extension $R[y;\sigma,\delta]$  is (left) noetherian if $R$ is a (left) noetherian and $\sigma$ is an automorphism. But if $\sigma$ is not automorphism then $R[y;\sigma,\delta]$  may not be (left) noetherian. (See \cite[Exercise 2P(b) and Theorem 2.6]{GoWa2}.) Likewise,  left double Ore extension
does not preserve noetherianess as seen in the following example.

\begin{example}
Let $R={\bf k}(t)$ be the quotient field of the polynomial ring ${\bf k}[t]$ and let $\sigma$ be the endomorphism on the polynomial ring $R[y_1]$ defined by
$$\sigma(f(t))=f(t^2),\ \ \sigma(y_1)=y_1$$
for all $f(t)\in R$. Note that $\sigma$ is injective.  Let $A$ be an iterated Ore extension
$A=R[y_1][y_2;\sigma]$. Then $A$ is a free left $R$-module with basis $\{y_1^iy_2^j|i,j=0,1,\ldots\}$ and
$$y_1f(t)=f(t)y_1,\ \ y_2f(t)=\sigma(f(t))y_2,\ \ y_1y_2=y_2y_1$$
for all $f(t)\in R$. Hence $A$ is a left double Ore extension of $R$ with a suitable DE-data. Since $A$ is an iterated Ore extension $A=R[y_2;\sigma|_{_{R}}][y_1]$ and $R[y_2;\sigma|_{_{R}}]$ is  left noetherian by  \cite[Exercise 2P(b)]{GoWa2}, $A$ is  left noetherian by \cite[Theorem~2.6]{GoWa2}. But it is easy to check that $A$ is not right noetherian. (See \cite[Exercise 2P(b)]{GoWa2}.)
\end{example}

\begin{cor}\label{ENVDPE2}
Let $R$ be a finitely generated Poisson algebra such that its Poisson enveloping algebra $R^e$ is an Artin-Schelter regular algebra and let
$A=R[x_1,x_2;\alpha,\nu]_p$ be a double Poisson-Ore extension of $R$. If $A^e$ is a connected graded algebra with degree
$\deg x_1=\deg x_2=\deg \text{d}x_1= \deg \text{d}x_2=1$ then $A^e$ is also an Artin-Schelter regular algebra and $\mathrm{gldim}(A^e)=\mathrm{gldim}(R^e )+4$.
\end{cor}

\begin{proof}
It follows immediately by Theorem~\ref{ENVDPE} and \cite[Theorem 0.2]{ZhZh}.
\end{proof}

\begin{cor}\label{GRENVDPE}
Let $R$ be a Poisson algebra that is finitely generated as an algebra and let
$$A=R[x_1,x_2;\alpha,\nu]_p$$ be a double Poisson-Ore extension of $R$. Then the Poisson enveloping algebra $A^e$ inherits the following properties from $\text{gr} _{\mathcal F}(R^e)$:
\begin{itemize}
\item[(1)] being a domain;
\item[(2)] being prime;
\item[(3)] being a maximal order;
\item[(4)] being Auslander-Gorenstein;
\item[(5)] having finite global dimension;
\item[(6)] having finite Krull dimension.
\end{itemize}
\end{cor}

\begin{proof}
By Lemma \ref{GRAD}, we know the commutative algebra $\text{gr}_{\mathcal F}(A^e)$ is isomorphic to the polynomial algebra over $\text{gr}_{\mathcal F}(R^e)$ with four variables. Then it is clear that $\text{gr}_{\mathcal F}(A^e)$ inherits properties from $\text{gr}_{\mathcal F}(R^e)$ regarding (1), (5) and (6). Moreover, (4) follows from \cite[Theorem 4.2]{EK}. Note that for commutative algebras, primeness is equivalent to domain and by \cite[Proposition 5.1.3]{McRo}, a noetherian commutative integral domain is a maximal order if and only if it is integrally closed. Hence (2) and (3) follow as well.

Further by Remark \ref{RemarkGR} (1), we know $\text{gr}_{\mathcal F}(R^e)$ is a finitely generated commutative algebra, Hence $\text{gr}_{\mathcal F}(A^e)$ is noetherian, then the filtration $\bigcup_{g\in G} \mathcal F_g=A^e$ is a Zariskian filtration. Thus $A^e$ inherits the similar properties (1)-(6) from $\text{gr}_{\mathcal F}(A^e)$ by the standard results of Zariskian filtration \cite{HO}.
\end{proof}

\noindent
{\bf Acknowledgments}
The second, third and fourth authors are grateful for the hospitality of the first author at Zhejiang Normal University summer 2016 during the time the project was started. The first and fourth authors are  supported by the National Natural Science Foundation of China (No. 11571316, No. 11001245 for the first author and No. 11301126, No. 11571316, No. 11671351 for the fourth author), and the first author is additionally supported by the Natural Science Foundation of Zhejiang Province (No. LY16A010003). The second author is supported by Chungnam National University Grant. The third author is supported by AMS-Simons travel grant.



\bibliographystyle{amsplain}

\providecommand{\bysame}{\leavevmode\hbox to3em{\hrulefill}\thinspace}
\providecommand{\MR}{\relax\ifhmode\unskip\space\fi MR }
\providecommand{\MRhref}[2]{%
  \href{http://www.ams.org/mathscinet-getitem?mr=#1}{#2}
}
\providecommand{\href}[2]{#2}

\end{document}